\documentclass[oneside,english]{amsart}
\usepackage[T1]{fontenc}
\usepackage[latin9]{inputenc}
\usepackage{mathtools}
\usepackage{amsbsy}
\usepackage{amstext}
\usepackage{amsthm}
\usepackage{amssymb}

\usepackage{graphicx}
\usepackage{float}
\usepackage{subfigure}

\usepackage{color}
\usepackage{xcolor}

\usepackage[unicode=true,pdfusetitle,
 bookmarks=true,bookmarksnumbered=false,bookmarksopen=false,
 breaklinks=false,pdfborder={0 0 1},backref=false,colorlinks=false]
 {hyperref}

\makeatletter

\numberwithin{equation}{section}
\theoremstyle{plain}
\newtheorem{thm}{\protect\theoremname}[section]
\theoremstyle{remark}
\newtheorem{rem}[thm]{\protect\remarkname}
\theoremstyle{definition}

\theoremstyle{plain}
\newtheorem{lem}[thm]{\protect\lemmaname}
\theoremstyle{plain}

\theoremstyle{plain}
\newtheorem{cor}[thm]{\protect\corollaryname}
\theoremstyle{plain}
\makeatother

\providecommand{\definitionname}{Definition}
\providecommand{\factname}{Fact}
\providecommand{\lemmaname}{Lemma}
\providecommand{\remarkname}{Remark}
\providecommand{\theoremname}{Theorem}
\providecommand{\corollaryname}{Corollary}

\begin{document}
\title[Discrete CKN inequalities and nonlinear elliptic equations]{Discrete Caffarelli-Kohn-Nirenberg inequalities and ground state
solutions to nonlinear elliptic equations }
\author{Fengwen Han and Ruowei Li}
\address{Fengwen Han: School of Mathematical and Statistics, Henan University,
475004 Kaifeng, Henan, China }
\email{fwhan@outlook.com}
\address{Ruowei Li: Department of Statistics and Data Science, National University of Singapore, 117546 Singapore}
\email{ruoweili@nus.edu.sg}
\begin{abstract}
In this paper, we prove the discrete Caffarelli-Kohn-Nirenberg inequalities
on the lattice $\mathbb{Z}^{N}$ ($N\geq 1$) in a broader range of parameters than
the classical continuous version \cite{CKN84}: 
\[
\parallel u\parallel_{\ell_{b}^{q}}\leq C(a,b,c,p,q,r,\theta,N)\parallel u\parallel_{D_{a}^{1,p}}^{\theta}\parallel u\parallel_{\ell_{c}^{r}}^{1-\theta},\:\forall u\in D_{a,0}^{1,p}(\mathbb{Z}^{N}) \cap \ell_c ^r(\mathbb{Z}^{N}),
\]
where $p,q,r>1,0\leq\theta\leq1$, 
%$\frac{1}{q^{\ast}}+\frac{b}{N},
$\frac{1}{p}+\frac{a}{N}>0,\frac{1}{r}+\frac{c}{N}>0,b\leq\theta a+(1-\theta)c,$
$\frac{1}{q^{\ast}}+\frac{b}{N}= \theta(\frac{1}{p}+\frac{a-1}{N})+(1-\theta)(\frac{1}{r}+\frac{c}{N})$ and $q\geq q^{\ast}$.
%and $q\geq q^{\ast}>0$.% and $0\leq a-\sigma$ if $\theta>0$. 
 For two special cases $\theta=1,a=0$ and $a=b=c=0$, by the discrete Schwarz rearrangement
established in \cite{HHH22}, we prove the existence
of extremal functions for the best constants in the supercritical
case $q>q^{\ast}$. As an application, we get positive ground state
solutions to the nonlinear elliptic equations.
\end{abstract}

\maketitle

\section{Introduction}

In \cite{CKN84},
Caffarelli, Kohn, and Nirenberg established the following Caffarelli-Kohn-Nirenberg (CKN for abbreviation) inequalities:
\begin{equation}
\parallel| x|^{b}u\parallel_{L^{q^{\ast}}}\leq C(a,b,c,p,r,\theta,N)\parallel| x|^{a}Du\parallel_{L^{p}}^{\theta}\parallel| x|^{c}u\parallel_{L^{r}}^{1-\theta}, \; \forall u\in C_{0}^{\infty}(\mathbb{R}^{N}), \label{eq:classical CKN ineqality}
\end{equation}
where $p,q^{\ast},q\geq 1,0\leq\theta\leq1,\frac{1}{q^{\ast}}+\frac{b}{N}>0,\frac{1}{p}+\frac{a}{N}>0,\frac{1}{r}+\frac{c}{N}>0,0\leq \theta a+(1-\theta)c-b\leq \theta$ and satisfying the dimensional balanced condition
\begin{equation}
\frac{1}{q^{\ast}}+\frac{b}{N}=\theta(\frac{1}{p}+\frac{a-1}{N})+(1-\theta)(\frac{1}{r}+\frac{c}{N}).\label{eq:condition-3(balance)}
\end{equation}
%See a version with higher orders derivatives in  by Lin.
See Lin \cite{L86} for a version with higher order derivatives.

%\begin{equation}
%N,p,r\geq1,q^{\ast}>0,0\leq\theta\leq1,\label{eq:condition-1}
%\end{equation}
%\begin{equation}
%\frac{1}{q^{\ast}}+\frac{b}{N},\frac{1}{p}+\frac{a}{N},\frac{1}{r}+\frac{c}{N}>0,b=\theta\sigma+(1-\theta)c,\label{eq:condition-2}
%\end{equation}

%\begin{equation}
%0\leq a-\sigma\;\text{if }\theta>0,\label{eq:condition-4(lower bud)}
%\end{equation}
%and 
%\begin{equation}
%a-\sigma\leq1\;\text{if \ensuremath{\theta>0\text{ and \ensuremath{\frac{1}{p}+\frac{a-1}{N}=\frac{1}{q^{\ast}}+\frac{b}{N}.}}}}\label{eq:condition-5(upper bud)}
%\end{equation}

Note that the CKN inequalities
(\ref{eq:classical CKN ineqality}) contain the classical Sobolev
inequality ($\theta=1,$ $a=b=0$), the Hardy inequality ($\theta=1,$
$a=0,b=-1$), the Hardy-Sobolev inequality ($\theta=1$) and the Gagliardo-Nirenberg
inequality ($a=b=c=0$) as special cases. Whether the best constant
can be obtained by some $u$, which is called the extremal function,
satisfies a prototype of more general nonlinear degenerate elliptic
equations from physical phenomena \cite{BE97,DL90} and references
therein. Thus it is a fundamental task to study the existence and nonexistence of extremal functions, as well as their
qualitative properties in inequalities (\ref{eq:classical CKN ineqality})
in the full parameter domain.

Much progress has been made for the case of $\theta=1,p=2$. In \cite{A76,T76},
the best constants and the minimizers for the Sobolev inequality ($a=b=0$)
were given by Aubin and Talenti via the Schwarz symmetrization. In
\cite{L83}, Lieb applied the Schwarz symmetrization to the case $a=0,-1<b<0,$
and gave the best constants and explicit minimizers. Since the symmetrization
method does not work for $a\neq0$, Chou and Chu \cite{CC93}
generalized the method of moving planes to conformally flat operators and
gave the explicit minimizers in the case
$a-1<b\leq a\leq0$. In \cite{L84-1,L85-2}, Lions established
the Concentration-Compactness method, and the case $a=0,-1<b\leq0$
was solved in \cite[Theorem 2.4]{L85-2}. Wang and Willem \cite{WW20}
gave a quantitative evaluation of the non-compactness of minimizing
sequences and established the compactness of all minimizing sequences
up to dilations provided $a-1<b\leq a<0$. The case $0<a<1$
and $b=0$ was treated by Caldiroli and Musina in \cite{CM99}, who
gave the existence of ground states by a family of compact approximating
problems. By converting the problem to an equivalent one defined on
$H^{1}(\mathbb{R}\times\mathbb{S}^{N-1})$, Catrina and Wang \cite{CW01}
proved minimizer is always achieved if $a-1<b<a$ or $a=b\leq0$ and
never achieved if $b=a-1$ or $a=b>0$. 
%And they observed symmetry breaking of minimizers occurs for $a>0$ and $b$ is sufficiently close to $a$. 
By the same transformation, Bartsch, Peng and Zhang
\cite{BPZ07} established some existence and non-existence results
with different geometries of the domain $\Omega$ for the case $p\neq2$.
When $\theta\neq1$, the CKN inequalities involves interpolation terms
which make the problem much tougher, see the review paper by Dolbeault
and Esteban \cite{DE12-1}. DelPino et al. studied the special cases $a=b=c=0$, $p=2,q^{\ast}=2p,r=p+1$ \cite{DD02}
and $p=r=2$, $-\frac{N-2}{2}<a$, $c=a-1$, $a-1\leq b<a$, $q=\frac{2N}{N-2+2(a-b)}$ \cite{DETT11,DE12-2}. %More results about the related progress, see for example \cite{MW12,TAX21,V20,FS03,SW12,ACM07,HK12} and references therein.

In recent years, people paid attention to the analysis on discrete
spaces, especially to the nonlinear equations
\cite{GLY16,HLY20,ZZ18,GHS22,HL21,HX21,HLW22,HLM22,L22,LW22}.
%\cite{GLY16,GLY17,HLY20,ZZ18,HSZ20,HS22,H22,G20,GHS22,HL21,HX21,HLW22,HLM22,L22,LW22}.
%Recently, there are many works concern the discrete Hardy and Rellich inequalities for $p=2$ as special cases of CKN inequalities \cite{KPP18-1,KPP18-2,KPP21,KL16,G22,RS09}.
As far as we know, there are no results about the general CKN inequalities on graphs. In this
article, we prove the general CKN inequalities, and obtain positive
ground state solutions of the nonlinear degenerate elliptic equations by the discrete Schwarz rearrangement established in \cite{HHH22}. 

A simple and undirected graph $G=(V,E)$ consists of the set of vertices
$V$ and the set of edges $E$. Two vertices $x,y$ are called neighbors,
denoted by $x\thicksim y,$ if there is an edge $e$ connecting $x$
and $y$, i.e. $e=\left\{ x,y\right\} \in E$. In this paper, we mainly
consider integer lattice graphs $\mathbb{Z}^{N}$ which serve as the discrete counterparts
of $\mathbb{R}^{N}$. It consists of the set of
vertices $V=\mathbb{Z}^{N}$ and the set of edges 
\[
E=\left\{ \left\{ x,y\right\} :x,y\in\mathbb{Z}^{N},\mathop{\sum\limits _{i=1}^{N}\lvert x_{i}-y_{i}\rvert=1}\right\} .
\]
The combinatorial distance $d$ is defined as $d(x,y)=\text{inf}{\left\{ k:x=x_{0}\sim\cdot\cdot\cdot\sim x_{k}=y\right\} }$,
and denote $d(x)=d(x,0)$ for abbreviation. In this paper we avoid the singularity of the origin by shifting the weight $d(x)^{}$ to $d(x)+1$, while other papers may assume $u(0)=0$. In fact, the CKN inequalities
with weights $(1+d(x))^{s}$  for all functions which vanish at infinity imply the inequalities with weights $d(x)^{s}$ for those functions $u$ satisfying $u(0)=0$. 
For $s\in \mathbb{R}$, we define a vertex weight function as
\begin{align*}
    \mu_{s}:\mathbb{Z}^{N} & \rightarrow(0,\infty),\\
    x & \mapsto (1+d(x))^{s}.
\end{align*}
And it can be generalized to a large class of weights $\mu $  satisfying $C^{-1} \mu_{s}(x)\leq \mu(x)\leq C \mu_{s}(x)$ as $x\to \infty$.
We denote by $\ell_{a}^{p}(\mathbb{Z}^{N})$ the $\ell^{p}$-summable
functions on $\mathbb{Z}^{N}$ with weight $\mu_{ap}(x)$, and by $D_{a,0}^{1,p}(\mathbb{Z}^{N})$
the completion of finitely supported functions in the $D_{a}^{1,p}$
norm 
\[
\lVert u\rVert_{D_{a}^{1,p}(\mathbb{Z}^{N})}\coloneqq\left(\sum\limits _{x\in\mathbb{Z}^{N}}\sum\limits _{y\sim x}\mu_{ap}(x)\lvert u(y)-u(x)\rvert^{p}\right)^{1/p}.
\]
Denote $D_{0}^{1,p}(\mathbb{Z}^{N})$ and $D^{1,p}$ if $a=0$ for
abbreviation.
%And define
%\[D^{1,p,r}(\mathbb{Z}^{N})\coloneqq \left\lbrace u\in {\ell^{r}%(\mathbb{Z}^{N})}: \lVert u\rVert_{D^{1,p}(\mathbb{Z}^{N})}< \infty %\right\rbrace 
%\]
For the case $\theta=1$ and $0\leq a-b\leq 1$, define
\[
D_{a}^{1,p}(\mathbb{Z}^{N})\coloneqq\left\{ u\in\ell_{b}^{q^{\ast}}(\mathbb{Z}^{N}):\lVert u\rVert_{D_{a}^{1,p}(\mathbb{Z}^{N})}<\infty\right\} ,\] 
where $q^{\ast}=\frac{Np}{N-p+p(a-b)}>1$ which is equivalent to (\ref{eq:condition-3(balance)}). In Theorem \ref{thm:the equivalence for Z^N_k=00003D00003D00003D1}, we prove the two definitions of $D_{a,0}^{1,p}(\mathbb{Z}^{N})$ and $D_{a}^{1,p}(\mathbb{Z}^{N})$ are equivalent as in \cite{HLM22}.

Duarte and Silva generalize the CKN inequalities to the cases of non-homogeneous weights $(1+|x|^2)^{\frac{s }{2}}$ \cite{MR4593125}.
By linearly interpolating $u$ defined on $\mathbb{Z}^{N}$, we can extend $u$ to the piecewise linear $\Bar{u}$ defined on $\mathbb{R}^{N}$, see (\ref{eq:extension def}) for detailed definitions. A similar
idea was introduced in \cite[Subsection 4.1]{RS09}. Next,
we prove the following discrete version.
\begin{thm}\label{thm:discrete_CKN_inequality}
If $N\geq 1,p,q,r>1,0\leq\theta\leq1, \frac{1}{p}+\frac{a}{N}>0,\frac{1}{r}+\frac{c}{N}>0, b \leq \theta a+(1-\theta)c, \frac{1}{q^{\ast}}+\frac{b}{N}=\theta(\frac{1}{p}+\frac{a-1}{N})+(1-\theta)(\frac{1}{r}+\frac{c}{N})
$ and $q= q^{\ast}$, then
\begin{equation}
\parallel u\parallel_{\ell_{b}^{q}}\leq C(a,b,c,p,q,r,\theta,N)\parallel u\parallel_{D_{a}^{1,p}}^{\theta}\parallel u\parallel_{\ell_{c}^{r}}^{1-\theta},\:\forall u\in D_{a,0}^{1,p}(\mathbb{Z}^{N}) \cap \ell_c ^r(\mathbb{Z}^{N}).\label{eq:CKN inequality}
\end{equation}
\end{thm} 
\begin{rem}

	(1) The weights $\mu_{s}$ can be generalized to a large class of weights $\mu $ satisfying $C^{-1} \mu_{s}(x)\leq \mu(x)\leq C \mu_{s}(x)$ as $x\to \infty$.

	(2) Since the discrete weight $\mu_{s}$ has no singularity
at the origin and by the monotonicity of $\mu_{s}$ with respect
to $s$, the continuous conditions $\frac{1}{q}+\frac{b}{N}>0$ and $\theta a+(1-\theta)c-b\leq \theta $ can be removed. 
%the range of parameters can be wider in discrete setting.

(3) Since $\ell^{p}(\mathbb{Z}^{N})$ embeds into $\ell^{q}(\mathbb{Z}^{N})$
for any $q>p$, one verifies that the
discrete CKN inequality (\ref{eq:CKN inequality}) holds
when $q\geq q^{\ast}$. It is called subcritical for $q<q^{\ast}$, critical
for $q=q^{\ast}$ and supercritical for $q>q^{\ast}$ for the CKN
inequality. Therefore, (\ref{eq:CKN inequality}) holds
in both critical and supercritical cases on $\mathbb{Z}^{N}$. 
%That is, the ``$=$'' in dimensional balanced condition (\ref{eq:condition-3(balance)}) can be replaced by ``$\leq$'' in the discrete setting. 

\end{rem}

Consider two special cases $\theta=1,a=0,q>q^{\ast}$ and $a=b=c=0,q>q^{\ast}$ respectively, and the optimal constants in the CKN inequality (\ref{eq:CKN inequality})
are given by 
\begin{equation}
S:=\inf_{\substack{u\in D_{0}^{1,p}(\mathbb{Z}^{N})\\
\lVert u\rVert_{\ell_{b}^{q}}=1
}
}\lVert u\rVert_{D^{1,p}}^{p},\label{eq: inf}
\end{equation}
where 
\begin{equation}
N\geq1,1< p<N,-1\leq b\leq0,q>q^{\ast}=\frac{Np}{N-p-pb}>1.\label{eq:condition-discrete}
\end{equation}
And 
\begin{equation}
K:=\inf_{\substack{u\in D_{0}^{1,p}(\mathbb{Z}^{N})\\
\lVert u\rVert_{\ell_{}^{q}}=1
}
}\lVert u\rVert_{D^{1,p}}^{\theta}\lVert u\rVert_{\ell^{r}}^{1-\theta},\label{eq: inf-2}
\end{equation}
where 
\begin{equation}
N\geq 1,p,r>1, 0\leq\theta\leq1, \frac{1}{q^{\ast}}=\theta(\frac{1}{p}-\frac{1}{N})+\frac{1-\theta}{r},q>q^{\ast}>1.
\label{eq:condition-discrete-2}
\end{equation}
%including p<N, hence satisfy the sobolev inequality and the spaces D1,p and D01,p are the same. Hence the corollary holds.

%Then the weights $\mu_{b}(x)=\mu_{b}(d(x))$ are radial nonincreasing.

In order to prove that the optimal constants are achieved, consider a minimizing
sequence $\{u_{n}\}\subset D_{0}^{1,p}(\mathbb{Z}^{N})$ satisfying
\begin{equation}
\lVert u_{n}\rVert_{\ell_{b}^{q}}=1,\;\lVert u_{n}\rVert_{D^{1,p}}^{p}\to S \text{\;or\;} \parallel u_{n}\parallel_{D^{1,p}}^{\theta}\parallel u_{n}\parallel_{\ell^{r}}^{1-\theta}\to K,n\to\infty.\label{eq:min seq}
\end{equation}
%\begin{equation}
%	\lVert u_{n}\rVert_{\ell_{b}^{q}}=1,\;\parallel u_{n}\parallel_{D^{1,p}}^{\theta}\parallel u_{n}\parallel_{\ell^{r}}^{1-\theta}\to K,n\to\infty.\label{eq:GN min seq}
%\end{equation}

Due to the lack of good symmetry on $\mathbb{Z}^{N}$, the moving plane
method \cite{CC93}, transformation method \cite{CW01} and approximation
method which needs Pohozaev-type identities \cite{CM99} cannot be
directly applied to the discrete setting. This variational problem
lacks translation invariance with the presence of a weight. Hence
the discrete Concentration-Compactness principle introduced in \cite{HL21,HLM22,L22} cannot be applied to the problem,
since we cannot rule out the vanishing limit by the translation trick.
Recently, the authors of \cite{HHH22} established a discrete
version of the Schwarz rearrangement on $\mathbb{Z}^{N}$ (see Section 2.2 for detailed
definitions), which provides a new idea to prove the existence of
optimizers of discrete functional inequalities. By
the discrete Hardy-Littlewood and P$\acute{\text{o}}$lya-Szeg$\ddot{\text{o}}$
inequalities \cite[Corollary 5.12, Theorem 5.15]{HHH22}, the minimizing
sequence can be restricted to the Schwarz symmetric function space
$\mathcal{S}(\mathbb{Z}^{N})$, then the minimizer exists by a compact
embedding result in the supercritical case $q>q^{\ast}$ \cite[Theorem 4.16]{HHH22}. %\textcolor{red}{
For more literature on rearrangement on graphs, we refer readers to \cite{frankl1989extremal, bollobas1991compressions, pruss98, weinstein1999excitation, burchard2006rearrangement, shlapentokh2010asymptotic, HH10, gupta2022symmetrization}.%}

We prove the following main results.
\begin{thm}
\label{thm:main1}If (\ref{eq:condition-discrete}) is satisfied,
then the variational problem (\ref{eq: inf}) admits a minimizer.
\end{thm}

\begin{rem}
(1) The minimizer obtained is nonnegative and satisfies a ``symmetric'' property, called Schwarz symmetric, see Section 2.2.

%(2) Note that (\ref{eq:condition-discrete}) implies $q>0$, and possibly $q<1$ when $b<-1$, which different from the continuous setting $b>-1$, $q= q^{\ast}\geq p\geq1.$ 
(2) The weights $\mu_{b}$ can be generalized to non-increasing radial weights $\mu $ that satisfy $C^{-1} \mu_{b}(x)\leq \mu(x)\leq C \mu_{b}(x)$ as $x\to \infty$.

(3) In the continuous setting, the parameter region of the existence
result is $-1<b\leq0$, $p=2$ and $q=2^{\ast}$, while the discrete
parameter region is $-1\leq b\leq0$ in the supercritical case $q>q^{\ast}.$ 

(4) For the case $a\neq0$, the Schwarz symmetric method cannot work
since the P$\acute{\text{o}}$lya-Szeg$\ddot{\text{o}}$ type inequality
with a weight is unknown. In the continuous setting, the parameter
region of the existence result $0<a-b<1$ or $a=b\leq0$ when $p=2$
is proved by the Concentration-Compactness principle \cite{WW20}
and the transformation method \cite{CW01}. Neither can be applied directly to $\mathbb{Z}^{N}$ because of the lack of scaling tricks.

(5) In the continuous setting, the authors in \cite{CW01} proved
that the optimal constant is never achieved when $b=a>0$ or $b=a-1$
by converting the problem to an equivalent one defined on $H^{1}(\mathbb{R}\times\mathbb{S}^{N-1})$,
which cannot be directly applied to $\mathbb{Z}^{N}$. So the nonexistence
problem on $\mathbb{Z}^{N}$ is open.

\end{rem}

\begin{thm}
\label{thm:main2}If (\ref{eq:condition-discrete-2}) is satisfied,
then the variational problem (\ref{eq: inf-2}) admits a minimizer.
\end{thm}

\begin{rem}
The difficulty of (\ref{eq: inf-2}) is the weak lower semi-continuity of the product term. %and prove the pointwise limit $u\in D_0^{1,p}(\mathbb{Z}^{N}) \cap \ell ^r(\mathbb{Z}^{N})$.
	We use the idea of exhaustion, while the continuous case applied
	the weak lower semi-continuity theorem \cite[Lemma 3.2]{P07} to an equivalent variational
	problem of constraints $\lVert u_{n}\rVert_{L^{q}(\mathbb{R}^{N})}=\lVert u_{n}\rVert_{L^{r}(\mathbb{R}^{N})}=1$
	by the homogeneous and dilation properties on $\mathbb{R}^{N}$, which
 cannot be applied on $\mathbb{Z}^{N}$ directly. 

\end{rem}

Recalling the continuous setting, the Euler-Lagrange equation of (\ref{eq: inf})
\begin{equation}
-\Delta u=\lvert x\rvert^{qb}u^{q-1}, x\in \mathbb{R}^{N}  \label{classical Hardy equation}
\end{equation}
is called the H$\acute{e}$non (resp. Hardy or Lane-Emden) equation if $b>0$ (resp. $b<0$ or $b=0$). For the supercritical case $q>q^{\ast}$ with $b>-1$, Ni \cite{Ni82,Ni86} got the existence result by a fixed point argument. As indicated by Mitidieri and Pohozaev in \cite[Theorem 6.1]{MP01}, Dancer, Du and Guo in \cite[Theorem 2.3]{DDG11} (see also Brezis and Cabr$\acute{e}$ \cite{BC98}), the condition $b>-1$ is necessary for the existence of solutions to (\ref{classical Hardy equation}). For the critical case $q=q^{\ast}$, Lions \cite[Section 2.4]{L85-2}, Wang and Willem \cite{WW20} proved the existence of solutions to (\ref{classical Hardy equation}) by the Concentration-Compactness principle. Catrina and Wang \cite{CW01} used the transformation method. For the subcritical case $1<q<q^{\ast}$, Reichel and Zou \cite{RZ00} proved the non-existence results with the refinement of the moving sphere method. This non-existence result was revisited by Phan and Souplet in \cite{PS12}. %Guo and Wan \cite{GW17} studied the equation (\ref{classical Hardy equation}) with weight $\lvert x\rvert^{a}$ in the left hand side. 
Ng$\hat{o}$ and Ye \cite{NY22} completed the classification of the solutions to higher order Hardy-H$\acute{e}$non equations.
%See \cite{LV16,LGZ06,V14,Le21} for a natural generalization to systems.
\\
%According to \cite{HW20}, we define the discrete $p$-Laplace of $u$
%for $p>1$ as
%\[
%\Delta_{p}u(x)\coloneqq\underset{y\sim x}{\sum}\lvert\nabla_{xy}u\rvert^{p-2}\nabla_{xy}u,
%\]
%for $p=1$ as
%\[
%\Delta_{1}u(x)\coloneqq\left\{ \sum\limits _{y\sim x}f_{xy}:f_{xy}=-f_{yx},f_{xy}\in\textrm{Sgn}(\nabla_{xy}u)\right\} \text{, Sgn\ensuremath{\left(t\right)}=\ensuremath{\begin{cases}
%\begin{array}{c}
%\left\{ 1\right\} ,\\{}
%[-1,1]\\
%\left\{ -1\right\} ,
%\end{array}, & \begin{array}{c}
%t>0,\\
%t=0,\\
%t<0.
%\end{array}\end{cases}}}
%\]

As an application, the minimizer obtained in Theorem \ref{thm:main1} is a positive ground state solution
to the following discrete degenerate elliptic equation.
\begin{cor}
\label{cor:1} If (\ref{eq:condition-discrete}) is satisfied, then there is a positive Schwarz symmetric ground state solution of the $p$-Laplace  Hardy
equation 
\begin{equation}
\Delta_{p}u+\mu_{qb}u^{q-1}=0,\;u\in D_{0}^{1,p}(\mathbb{Z}^{N}),\label{eq:EL equation}
\end{equation}
where the discrete $p$-Laplace
$\Delta_{p}u(x)\coloneqq\underset{y\sim x}{\sum}\lvert u(y)-u(x)\rvert^{p-2}(u(y)-u(x)).
$
\end{cor}

%\begin{rem}
%(1) Note that $b\leq0$, $q>0$, and possibly $q<1$ when $b<-1$, which different from the continuous setting $b>-1$, $q\geq q^{\ast}\geq p\geq1.$ 
%The  ground state solution is Schwarz symmetric, see Section 4.

%\end{rem}
Similarly, we get a positive ground state solution to the Euler-Lagrange
equation of (\ref{eq: inf-2}), which
has many studies in the discrete setting \cite{HLW22} and continuous setting \cite{BFV14,Z21,DD02}.
\begin{cor}
	\label{cor: GN_EL equation}If (\ref{eq:condition-discrete-2}) is satisfied,	then there is a positive Schwarz symmetric ground state solution of
	the equation  
	\begin{equation}
		\lambda_1 \Delta_{p}u- \lambda_2 u^{r-1}+  u^{q-1}=0,\; u\in D_0^{1,p}(\mathbb{Z}^{N}) \cap \ell ^r(\mathbb{Z}^{N}), \label{eq:EL equation of GN}
	\end{equation}
	where $\lambda_1,\lambda_2 >0$ are determined by parameters $N,p,q,r,\theta$
	and $u$. In particular, if $\theta=0$, then $\lambda_1\Delta_{p}u$ term
	is omitted, and if $\theta=1$, then $\lambda_2 u^{r-1}$ term is omitted.
\end{cor}

\begin{rem}

According to the results in continuous cases,
we conjecture that (\ref{eq:EL equation}) and (\ref{eq:EL equation of GN}) have positive solutions
when $q=q^{\ast}$, and non-negative solutions are trivial when
$q<q^{\ast}$.
\end{rem}

The paper is organized as follows. In Section 2, we recall some basic
notations and introduce the discrete Schwarz rearrangement. In Section 3, we prove the discrete CKN inequality
using an extension method and the equivalence of Sobolev spaces. In
Section 4, we prove Theorem \ref{thm:main1}, Theorem \ref{thm:main2}, Corollary \ref{cor:1} and Corollary \ref{cor: GN_EL equation}. In this paper, we use $a\lesssim b$ to denote $a\leq Cb$ for some
$C>0$ and $a\thickapprox b$ to denote $a\lesssim b\lesssim a$.

\section{Preliminaries}

\subsection{Graphs and basic notations.} Consider integer lattice graph $\mathbb{Z}^{N}$, which is the graph
consisting of the set of vertices $V=\mathbb{Z}^{N}$ and
the set of edges 
\[
E=\left\{ \left\{ x,y\right\} :x,y\in\mathbb{Z}^{N},\mathop{\sum\limits _{i=1}^{N}\lvert x_{i}-y_{i}\rvert=1}\right\} .
\]
The combinatorial distance $d$ is defined as $d(x,y)=\text{inf}{\left\{ k:x=x_{0}\sim\cdot\cdot\cdot\sim x_{k}=y\right\} }$,
i.e. the length of the shortest path connecting $x$ and $y$ by assigning
each edge of length one, and denote $d(x)=d(x,0)$ for abbreviation.

We denote the space of functions on $\mathbb{Z}^{N}$ by $C(\mathbb{Z}^{N})$.
For $u\in C(\mathbb{Z}^{N})$, its support set is defined as $\textrm{supp}(u)\coloneqq\{x\in\mathbb{Z}^{N}:u(x)\neq0\}$.
Let $C_{c}(\mathbb{Z}^{N})$ be the set of all functions with finite
support, and $$C_0(\mathbb{Z}^{N}):=\{u \in C(\mathbb{Z}^{N}): |\{x : |u(x)| > t \}|<+\infty,\ \forall t > 0\}$$ be the space of functions that vanish at infinity. For $s\in \mathbb{R}$, we define a vertex weight function as
\begin{align*}
    \mu_{s}:\mathbb{Z}^{N} & \rightarrow(0,\infty),\\
    x & \mapsto (1+d(x))^{s}.
\end{align*}
For any $u\in C(\mathbb{Z}^{N})$,
the $\ell_{a}^{p}$ norm of $u$ is defined
as 
\[
\lVert u\rVert_{\ell_{a}^{p}(\mathbb{Z}^{N})}\coloneqq\begin{cases}
\left(\underset{x\in\mathbb{Z}^{N}}{\sum}\mu_{ap}(x)\lvert u(x)\rvert^{p}\right)^{1/p} & \text{\ensuremath{1\leq p<\infty,}}\\
\underset{x\in\mathbb{Z}^{N}}{\sup}\mu_{a}(x)\lvert u(x)\rvert & p=\infty.
\end{cases}
\]
The $\ell_{a}^{p}(\mathbb{Z}^{N})$ space is defined as 
\[
\ell_{a}^{p}(\mathbb{Z}^{N})\coloneqq\left\{ u\in C(\mathbb{Z}^{N}):\lVert u\rVert_{{\ell_{a}^{p}(\mathbb{Z}^{N})}}<\infty\right\} .
\]
For any $u\in C(\mathbb{Z}^{N})$, define the difference operator
for any $x\sim y$ as

\[
\nabla_{xy}u=u(y)-u(x).
\]
Let 
\[
\lvert\nabla u(x)\rvert_{p}\coloneqq\left(\sum\limits _{y\sim x}\lvert\nabla_{xy}u\rvert^{p}\right)^{1/p}
\]
be the $p$-norm of the gradient of $u$ at $x$.

The $D_{a}^{1,p}$ norm of $u$ is given by 
\[
\lVert u\rVert_{D_{a}^{1,p}(\mathbb{Z}^{N})}=\lVert\mu_{a}\nabla u\rVert_{\ell^{p}(\mathbb{Z}^{N})}\coloneqq\left(\sum\limits _{x\in\mathbb{Z}^{N}}\sum\limits _{y\sim x}\mu_{ap}(x)\lvert\nabla_{xy}u\rvert^{p}\right)^{1/p},
\]
then $D_{a,0}^{1,p}(\mathbb{Z}^{N})$ is the completion of $C_{c}\left(\mathbb{Z}^{N}\right)$
in $D_{a}^{1,p}$ norm. And denote $D_{0}^{1,p}(\mathbb{Z}^{N})$
and $D^{1,p}$ if $a=0$ for abbreviation.

\subsection{Discrete Schwarz rearrangement.} The discrete Schwarz rearrangement developed in \cite{HHH22} is crucial for the proof of the main results, so we give a brief introduction here. Let $$C_0^+(\mathbb{Z}^{N}):=\{u \in C_0(\mathbb{Z}^{N}): u\text{ is non-negative}\}$$ be the set of functions that are admissible for discrete Schwarz rearrangement. The discrete Schwarz rearrangement in $\mathbb{Z}^{N}$ is defined via the following 3 parts.

\underline{Part 1}: Define one-dimensional rearrangements.\\
Let $f \in C_0^+(\mathbb{Z})$ or $C_0^+(\mathrm{\mathbb{Z}+\frac{1}{2}})$ be non-negative with function values: $f_1\ge f_2 \ge \cdots $, and define one-dimensional rearrangement via
\begin{equation*}
    f^{\star}(x)=
    \begin{cases}
    f_{1-2x}, \qquad & x\le 0;\\
    f_{2x}, & x>0.
    \end{cases}
\end{equation*}

\begin{figure}[h]
    \centering
    \subfigure[$f^{\star}$ on $\mathbb{Z}$]{
        \includegraphics[width=0.52\textwidth]{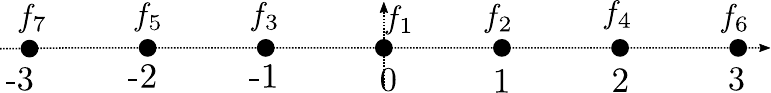}}\\
    \subfigure[$f^{\star}$ on $\mathbb{Z}+\frac{1}{2}$]{
        \includegraphics[width=0.52\textwidth]{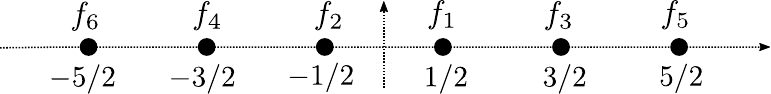}}
    \caption{Two types of one-dimensional rearrangements}
    
\end{figure}

\underline{Part 2}: Define one-step rearrangements.\\
Let $e_i$ be the $i$-th standard $N$-dimensional unit vector, and let
$$\Vec{E}=\{e_i,\frac{e_i-e_j}{2},\frac{e_i+e_j}{2}: \ 1\le i< j \le N\}.$$
For each $e\in \Vec{E}$, there is a natural equivalence relation in $\mathbb{Z}^N$ defined via $x \overset{e}{\sim} y$ if and only if $x-y$ is parallel to $e$ in $\mathbb{R}^N$. Then there is a unique corresponding partition $\mathbb{Z}^N=\bigsqcup_{\alpha\in I_{e}}V_e^{\alpha}$, %\textcolor{red}{
where $I_e$ is the index set and each equivalence class $V_e^{\alpha}$ is the restriction of $\mathbb{Z}^N$ on a line parallel to $e$. Moreover, for any equivalence class $V_e^{\alpha}$ with respect to $\overset{e}{\sim}$, the map 
\begin{equation*}
    x \mapsto <e,x>
\end{equation*}
is a bijection to $\mathbb{Z}$ or $\mathbb{Z}+\frac{1}{2}$. Then $(u|_{V_e^{\alpha}})^{\star}$ is well defined and the one-step rearrangement of $u$ with respect to $e$ is defined via 
\begin{equation*}
    (R_e u)|_{V_e^{\alpha}}=(u|_{V_e^{\alpha}})^{\star}, \forall u\in C_0^+(\mathbb{Z}^N).
\end{equation*}
Generally speaking, $\mathbb{Z}^{N}$ is decomposed into the union of one-dimensional $\mathbb{Z}$s or $(\mathbb{Z}+\frac{1}{2})$s with respect to directions in $\Vec{E}$. Then one-step rearrangement is defined by rearrange $u$ on each $\mathbb{Z}$ or $\mathbb{Z}+\frac{1}{2}$.
\begin{figure}[H]
    \centering
    \setcounter{subfigure}{0}
    \subfigure[]{\includegraphics[width=0.24\textwidth]{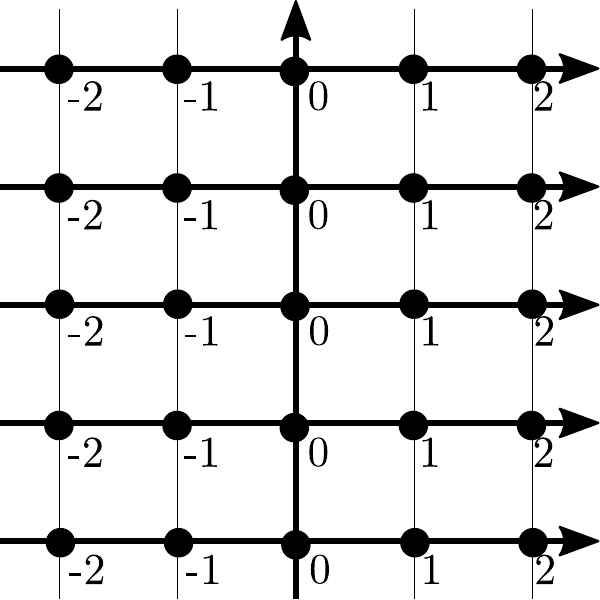}}\qquad
    \subfigure[]{\includegraphics[width=0.24\textwidth]{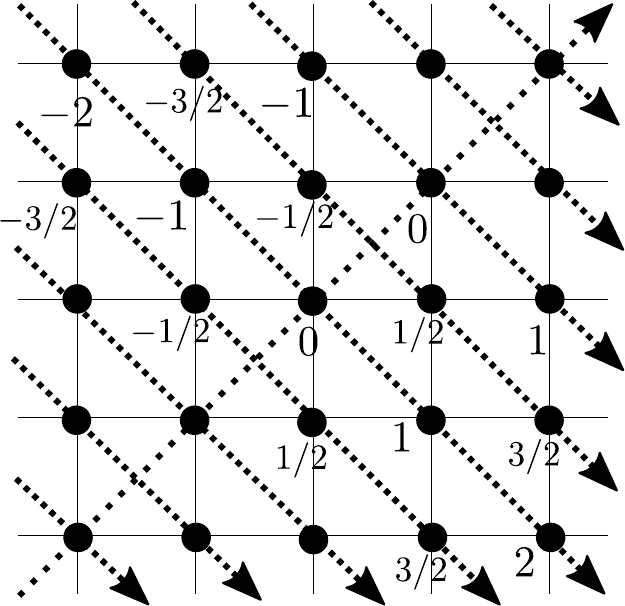}}
    \caption{Decomposition of $\mathbb{Z}^2$ w.r.t. directions $e_1$ and $\frac{e_1-e_2}{2}$}
    \label{fig-V_2}
\end{figure}

\underline{Part 3}: Define Schwarz rearrangement in $\mathbb{Z}^N$.\\
Let $R^k$ be the iteration of $k$ one-step rearrangements, i.e.
$$R^k u:=\underbrace{\cdots \circ R_{e_1} \circ \cdots \circ R_{\frac{e_1+e_2}{2}}\circ \cdots \circ R_{\frac{e_1-e_2}{2}} \circ \cdots \circ R_{e_2} \circ R_{e_1}}_k u.$$

The \emph{discrete Schwarz rearrangement} of an admissible function $u\in C_0^+(\mathbb{Z}^N)$ is defined via
\begin{equation}
    \label{eq:def_of_DSR}
    u^{\star}=R_{\mathbb{Z}^N} u(x):= \lim_{k\to \infty} R^k u(x), \ x\in \mathbb{Z}^N.
\end{equation}

The limit in \eqref{eq:def_of_DSR} exists, so $u^{\star}$ is well-defined; see \cite{HHH22}. A function $u\in C_{0}^{+}(\mathbb{Z}^{N})$ is called Schwarz symmetric
if $u=u^{\star}$. Let $\mathcal{S}(\mathbb{Z}^{N})\coloneqq\left\{ u\in C_{0}^{+}(\mathbb{Z}^{N}):u=u^{\star}\right\} $
be the set of all Schwarz symmetric functions on $\mathbb{Z}^{N}$.
The following results are crucial for our next proof. 

\begin{lem}[Discrete Hardy-Littlewood inequality; \cite{HHH22}, Corollary 5.12]
\label{lem:rearrangement inequalities} Let $u,v\in C_{0}^{+}(\mathbb{Z}^{N})$, then 
\[
\sum_{x \in \mathbb{Z}^{N}}u(x)v(x)\leq\sum_{x \in \mathbb{Z}^{N}}u^{\star}(x)v^{\star}(x). 
\]
\end{lem}

\begin{lem}[Discrete P$\acute{\text{o}}$lya-Szeg$\ddot{\text{o}}$ inequality; \cite{HHH22}, Theorem 5.15]
\label{lem:Polya-Szego_inequality}
Let $u\in\ell^{p}(\mathbb{Z}^{N})$ be non-negative for $p\geq1$, then 
\[
\lVert\nabla u^{\star}\rVert_{p}\leq\lVert\nabla u\rVert_{p}. 
\]
\end{lem}

\begin{lem}[\cite{HHH22}, Theorem 4.16]
\label{lem:compact embedding}Let $\left\{ u_{n}\right\} $ be a sequence
in $\mathcal{S}(\mathbb{Z}^{N})\cap\ell^{p}(\mathbb{Z}^{N})$ with
$\lVert u_{n}\rVert_{\ell^{p}}\leq1$ for $p\geq1$. Then for all $q>p$,
there exists a subsequence of $\left\{ u_{n}\right\} $ convergent
in $\ell^{q}(\mathbb{Z}^{N}).$
\end{lem}

\section{The discrete CKN inequality }

We introduce an extension method to extend a function $u$ defined
on $\mathbb{Z}^{N}$ to $\bar{u}$ defined on $\mathbb{R}^{N}$, and
a similar method was introduced in \cite[Subsection 4.1]{RS09}. Then we prove an equivalence of spaces.
\subsection{The discrete CKN inequality}
Given any $u\in C(\mathbb{Z}^{N})$, in every elementary cubic cell $Q$, each point $x\in Q$ can be expressed uniquely as a convex combination of its $2^{N}$ vertices $v^{i}$, i.e. $x=\sum_{i=1}^{2N} c_{i}(x)v^{i}$ with $c_{i}(x)\geq 0$ and $\sum_{i=1}^{2N} c_{i}(x)=1$. Then we define a piecewise linear function $\bar{u}\in C(\mathbb{R}^{N})$ as
\begin{equation} 
	\bar{u}(x)=\sum_{i=1}^{2N} c_{i}(x)u(v^{i}).
	\label{eq:extension def}
\end{equation}
%denoted by $\bar{u}$ for short. 
Let $\mathcal{L}(Q)$ be the space of such piecewise linear functions $\bar{u}$ on the cell $Q$. Clearly, dim$\mathcal{L}(Q)=2^{N}$.
By the definition of $\bar{u}$, we get the following two estimates
of $\bar{u}$ and $u$ .
\begin{lem}
\label{lem: estimate of norm}If $u\in\ell^{p}(\mathbb{Z}^{N})$ for
$p\geq 1$, then $\bar{u}\in L^{p}(\mathbb{R}^{N})$ and their norms
 equivalent: 
\begin{equation}
\lVert\bar{u}\rVert_{L^{p}(\mathbb{R}^{N})}\thickapprox\lVert u\rVert_{\ell^{p}(\mathbb{Z}^{N})}.\label{eq:norms equivalent}
\end{equation}
\end{lem}

\begin{proof}
By the argument above, on $\mathcal{L}(Q)$ we consider two norms:
\[
\lVert u\rVert_{\ell^{p}(Q)}=\left(\underset{y\in\left\{ 0,1\right\} ^{N}}{\sum}\lvert\bar{u}(y)\rvert^{p}\right)^{1/p},
\]
\[
\lVert\bar{u}\rVert_{L^{p}(Q)}=\left(\int_{Q}|\bar{u}(x)|^{p}\text{d}x\right)^{1/p}.
\]
Then $\lVert u\rVert_{\ell^{p}(Q)}\thickapprox\lVert\bar{u}\rVert_{L^{p}(Q)}$ by the equivalence of norms in a finite dimensional linear space.
By summing similar inequalities over all cells $Q+y$ with $y\in\mathbb{Z}^{N}$,
we obtain (\ref{eq:norms equivalent}). Note that the extension operator is linear and Lipschitz continuous with respect to norms $\ell^{p}$ and $L^{p}$. Since $C_{c}(\mathbb{Z}^{N})$
is dense in $\ell^{p}(\mathbb{Z}^{N})$, we have $\bar{u}\in L^{p}(\mathbb{R}^{N})$ for any $u\in\ell^{p}(\mathbb{Z}^{N})$.
\end{proof}
\begin{lem}
\label{lem: estimate of energy}If $u\in D_{0}^{1,p}(\mathbb{Z}^{N})$
for $p\geq 1$, then $\bar{u}\in D_{0}^{1,p}(\mathbb{R}^{N})$ which is the
completion of $C_{c}^{\infty}(\mathbb{R}^{N})$ with $\lVert\nabla\cdot\rVert_{L^{p}}$
norm, and their $p$-Sobolev energy equivalent: 
\begin{equation}
\lVert\nabla\bar{u}\rVert_{L^{p}(\mathbb{R}^{N})}\thickapprox\lVert u\rVert_{D^{1,p}(\mathbb{Z}^{N})}.\label{eq: p-sobolev energy equivalent}
\end{equation}
\end{lem}

\begin{proof}
Define $\mathcal{I}(Q)$ as the space of constant
functions on $Q$. Let $\mathcal{\bar{L}}(Q)\coloneqq\mathcal{L}(Q)/\mathcal{I}(Q)$,
and we consider two norms in $\mathcal{\bar{L}}(Q)$: 
\[
\lVert u\rVert_{D^{1,p}(Q)}=\left(\underset{\underset{x,y\in\left\{ 0,1\right\} ^{N}}{x\sim y}}{\sum}\lvert\bar{u}(y)-\bar{u}(x)\rvert^{p}\right)^{1/p},
\]
\[
\lVert\nabla\bar{u}\rVert_{L^{p}(Q)}=\left(\int_{Q}|\nabla\bar{u}(x)|^{p}\text{d}x\right)^{1/p}.
\]
By the equivalence of norms in the finite dimensional linear space $\mathcal{\bar{L}}(Q)$, we get $\lVert\nabla\bar{u}\rVert_{L^{p}(Q)}\thickapprox\lVert u\rVert_{D^{1,p}(Q)}$.
By summing similar inequalities over all cells $Q+y$ with $y\in\mathbb{Z}^{N}$,
we obtain (\ref{eq: p-sobolev energy equivalent}). Note that the extension operator is linear and Lipschitz continuous with respect to norms $\ell^{p}$ and $L^{p}$. 
Since $C_{c}(\mathbb{Z}^{N})$ is dense in $D_{0}^{1,p}(\mathbb{Z}^{N})$, then $\bar{u}\in D_{0}^{1,p}(\mathbb{R}^{N})$ for any $u\in D_{0}^{1,p}(\mathbb{Z}^{N})$.
\end{proof}

Then we introduce the CKN inequalities with non-homogeneous weights in $\mathbb{R}^N$ \cite{MR4593125}.%CKN inequalities for non-homogeneous weights:
%$\frac{1}{q}\leq \frac{\theta}{p}+\frac{1-\theta}{r}$
\begin{thm}[CKN inequalities of non-homogeneous weights; \cite{MR4593125}, Theorem 1.7]
If $p,q,r>1,0\leq\theta\leq1$, $a\in (-N/p,N/p')$, $c\in (-N/r,N/r')$,
$b>-N/q$, $0\leq\theta a+(1-\theta)c-b\leq \theta$ and $$\theta(\frac{1}{p}-\frac{1}{N})+\frac{1-\theta}{r}\leq \frac{1}{q}\leq \theta (\frac{1}{p}+\frac{a-1}{N})+(1-\theta)(\frac{1}{r}+\frac{c}{N})-\frac{b}{N},$$ then it follows that 
\begin{equation}
\parallel \langle x \rangle ^b u\parallel_{L^{q}}\leq C(a,b,c,p,q,r,\theta,N)\parallel \langle x \rangle ^a u\parallel_{D^{1,p}}^{\theta}\parallel \langle x \rangle ^c u\parallel_{L^{r}}^{1-\theta},\:\forall u\in \mathbb{S}(\mathbb{R}^{N}), \label{eq: CKN inequalities of non-homogeneous weights}
\end{equation}
where $\langle x \rangle ^a=(1+|x|^2)^{\frac{a}{2}}$ and $\mathbb{S}(\mathbb{R}^{N})$ is the space of Schwartz functions (i.e. smooth, rapidly decreasing function).
\end{thm}

%$\frac{1}{q^{\ast}}+\frac{b}{N},
%$\frac{1}{p}+\frac{a}{N},\frac{1}{r}+\frac{c}{N}>0,b\leq\theta a+(1-\theta)c,$

Then we are ready to prove the discrete version of CKN inequality.

\begin{proof}[Proof of Theorem~\ref{thm:discrete_CKN_inequality}]

By the equivalence in Lemma \ref{lem: estimate of norm}, Lemma \ref{lem: estimate of energy} and the CKN inequality
(\ref{eq: CKN inequalities of non-homogeneous weights}) with $q=q^{\ast}$, 
 we get the discrete CKN inequalities (\ref{eq:CKN inequality}). By incorporating the discrete nature, we expand the range of the parameters.
\end{proof}

%\begin{rem}
%	(1) The weights $\mu_{s}$ can be generalized to a large class of weights $\mu $ satisfying $\mu(x)\approx \mu_{s}(x)$ as $x\to \infty$.	

%(2) The range of param
%(\ref{eq:condition-5(upper bud)}) can be removed by the monotonicity
%of $\mu_{s}$ with respect to $s$. 

%(3) Since $\ell^{p}(\mathbb{Z}^{N})$ embeds into $\ell^{q}(\mathbb{Z}^{N})$
%for any $q>p$, one verifies that the discrete CKN inequality (\ref{eq:CKN inequality})
%holds when $q\geq q^{\ast}$. That is, the ``$=$'' in dimensional balanced condition (\ref{eq:condition-3(balance)}) can be replaced by ``$\leq$'' in the discrete setting.

%\end{rem}
\subsection{The equivalence of Sobolev spaces}
Consider the case $\theta=1$ and $0\leq a-b\leq1$, as our previous work \cite{HLM22} we define
\[
D_{a}^{1,p}(\mathbb{Z}^{N})\coloneqq\left\{ u\in\ell_{b}^{q^{\ast}}(\mathbb{Z}^{N}):\lVert u\rVert_{D_{a}^{1,p}(\mathbb{Z}^{N})}<\infty\right\} ,\] 
where $q^{\ast}=\frac{Np}{N-p+p(a-b)}$ which is equivalent to (\ref{eq:condition-3(balance)}). Note that $D_{a,0}^{1,p}(\mathbb{Z}^{N})$ is
the completion of finitely supported functions in the $D_{a}^{1,p}$
norm.
Next we prove that the two spaces are equivalent.
%$D_{a,0}^{1,p}(\mathbb{Z}^{N})=D_{a}^{1,p}(\mathbb{Z}^{N}).$

%Theorem~\ref{thm:the equivalence for Z^N_k=00003D00003D00003D1}.

\begin{thm}
\label{thm:the equivalence for Z^N_k=00003D00003D00003D1}If $N\geq 1,p>1, \frac{1}{p}+\frac{a}{N}>0, 0\leq a-b\leq 1, q^{\ast}=\frac{Np}{N-p+p(a-b)}>1$, then 
\[
D_{a,0}^{1,p}(\mathbb{Z}^{N})=D_{a}^{1,p}(\mathbb{Z}^{N}).
\]
\end{thm}

\begin{proof}
%For any $u\in D_{a,0}^{1,p}(\mathbb{Z}^{N})$, there exists a sequence $\left\{ u_{n}\right\} \subset C_{c}\left(\mathbb{Z}^{N}\right)$ such that 
%\[
%u_{n}\to u\text{ in }D_{a}^{1,p}.
%\]
%And $u\in\ell_{b}^{q^{\ast}}(\mathbb{Z}^{N})$ by the CKN inequality (\ref{eq:CKN inequality}). 
First $D_{a,0}^{1,p}(\mathbb{Z}^{N})\subseteq D_{a}^{1,p}(\mathbb{Z}^{N})$ follows from the CKN inequality (\ref{eq:CKN inequality}). 
For the other direction, the key is to find suitable cutoff functions
$\eta_{n}(x): C_{c}\left(\mathbb{Z}^{N}\right)\to [0,1]$. For any $u\in D_{a}^{1,p}(\mathbb{Z}^{N})$,
set $u_{n}\coloneqq u\eta_{n}$.

If $0\leq a-b<1$, then by H$\ddot{\text{o}}$lder inequality,
\begin{align}
\lVert u_{n}-u\rVert_{D_{a}^{1,p}(\mathbb{Z}^{N})}^{p} & =\sum\limits _{x\in\mathbb{Z}^{N}}\sum\limits _{y\sim x}\mu_{ap}(x)\lvert\nabla_{xy}(u\eta_{n})-\nabla_{xy}u\rvert^{p}\nonumber \\
 & =\sum\limits _{x\in\mathbb{Z}^{N}}\sum\limits _{y\sim x}\mu_{ap}(x)\lvert\eta_{n}(y)\nabla_{xy}u+u(x)\nabla_{xy}\eta_{n}-\nabla_{xy}u\rvert^{p}\label{eq:k=00003D00003D00003D1-1}\\
 & \lesssim\sum\limits _{x\in\mathbb{Z}^{N}}\mu_{ap}(x)\lvert\nabla u(x)|_{p}^{p}\underset{y\sim x}{\text{max}}|\eta_{n}(y)-1|^{p}+\lVert\nabla\eta_{n}\rVert_{\ell_{a-b}^{\frac{N}{b-a+1}}}^{p}\lVert u\rVert_{\ell_{b}^{q^{\ast}}}^{p}.\nonumber 
\end{align}
If the cutoff functions satisfy 
\begin{equation}
\eta_{n}\to1\text{ pointwise on \ensuremath{\mathbb{Z}^{N}}},\;\lVert\nabla\eta_{n}\rVert_{\ell_{a-b}^{\frac{N}{b-a+1}}}\rightarrow0,\label{eq: cut off fun req 2}
\end{equation}
then the other direction follows from the dominated convergence theorem.

Let $r>10$, and $R\gg r$ be large enough. Define 
\[
\eta(x)\coloneqq1\wedge\frac{\textrm{log\ensuremath{R}}-\text{log\ensuremath{|x|}}}{\textrm{log\ensuremath{R}}-\text{log\ensuremath{r}}}\vee0,
\]
where $|x|$ stands for the Euclidean distance. Then

\begin{align*}
\lVert\nabla\eta\rVert_{\ell_{a-b}^{\frac{N}{b-a+1}}(\mathbb{Z}^{N})}^{\frac{N}{b-a+1}} & =\sum\limits _{r\leq| x|\leq R}\sum\limits _{y\sim x}\mu_{\frac{(a-b)N}{b-a+1}}\lvert\nabla_{xy}\eta\rvert^{\frac{N}{b-a+1}}\\
 & \lesssim\left(\textrm{log\ensuremath{\frac{R}{r}}}\right)^{-\frac{N}{b-a+1}}\sum\limits _{r-1\leq| x|\leq R+1}\sum\limits _{y\sim x}\mu_{\frac{(a-b)N}{b-a+1}}\lvert\textrm{log\ensuremath{| x|}}-\text{log\ensuremath{\ensuremath{| y|}}}\rvert^{\frac{N}{b-a+1}}\\
 & \lesssim\left(\textrm{log\ensuremath{\frac{R}{r}}}\right)^{-\frac{N}{b-a+1}}\sum\limits _{r-1\leq| x|\leq R+1}\mu_{\frac{(a-b)N}{b-a+1}}\lvert x\rvert^{-\frac{N}{b-a+1}},
\end{align*}
where the second inequality follows from the mean value theorem. Next, we estimate the summation on $\mathbb{Z}^{N}$
by the integral on $\mathbb{R}^{\mathit{N}}$ as follows 
\begin{align*}
\sum\limits _{r-1\leq| x|\leq R+1}\mu_{\frac{(a-b)N}{b-a+1}}\lvert x\rvert^{-\frac{N}{b-a+1}} & \lesssim\sum\limits _{r-2\leq| x|\leq R+2}\int\limits _{S_{x}(\frac{1}{2})}| x|^{\frac{(a-b)N}{b-a+1}}\lvert x\rvert^{-\frac{N}{b-a+1}}\text{d\ensuremath{t}}\\
 & \lesssim\sum\limits _{r-2\leq| x|\leq R+2}\int\limits _{S_{x}(\frac{1}{2})}\lvert t\rvert^{-N}\text{d\ensuremath{t}}\\
 & \lesssim\int\limits _{B_{R+3}\setminus B_{r-3}}\lvert t\rvert^{-N}\text{d\ensuremath{t}},
\end{align*}
where $S_{x}(\frac{1}{2})\coloneqq\left\{ t\in\mathbb{R}^{\mathit{N}}:\lvert t_{i}-x_{i}\rvert<\frac{1}{2},1\leq i\leq N\right\} $
is the Euclidean cube, and $B_{r}$ is the Euclidean ball with radius
of $r$ and centered at the origin. Hence, 
\[
\lVert\mu_{a-b}\nabla\eta\rVert_{\ell^{\frac{N}{b-a+1}}}^{\frac{N}{b-a+1}}\lesssim\left(\textrm{log\ensuremath{\frac{R}{r}}}\right)^{-\frac{N}{b-a+1}}\text{log}\ensuremath{\frac{R}{r}}=O\left(\textrm{\ensuremath{\left(\text{log}\ensuremath{\frac{R}{r}}\right)}}^{1-\frac{N}{b-a+1}}\right).
\]
Note that $1-\frac{N}{b-a+1}<0$ since $1-N\leq a-b<1$. Then
for fixed $r$, letting $R\to\infty$, then $\lVert\mu_{a-b}\nabla\eta\rVert_{\ell^{\frac{N}{b-a+1}}}\to0$.
Hence $u\in D_{a,0}^{1,p}(\mathbb{Z}^{N})$ by (\ref{eq:k=00003D00003D00003D1-1}).

If $a-b=1$, we need to prove $\lVert\nabla\eta\rVert_{\ell_{1}^{\infty}}\rightarrow0$
instead of (\ref{eq: cut off fun req 2}). By the same definition
of $\eta$, we have 
\[
\lVert\nabla\eta\rVert_{\ell_{1}^{\infty}}\lesssim\underset{r-1\leq| x|\leq R+1}{\sup}\sum\limits _{y\sim x}\mu_{1}(x)\left(\textrm{log\ensuremath{\frac{R}{r}}}\right)^{-1}\lvert\textrm{log\ensuremath{| x|}}-\text{log\ensuremath{\ensuremath{| y|}}}\rvert=O\left(\textrm{\ensuremath{\left(\text{log}\ensuremath{\frac{R}{r}}\right)}}^{-1}\right),
\]
which goes to zero as $R\to\infty$. Thus we prove the result.
\end{proof}

\section{The existence of minimizers}

Now we are ready to prove Theorem \ref{thm:main1}.
\begin{proof}[Proof of Theorem \ref{thm:main1}]
Note that radial nonincreasing weight functions $\mu_{b}\in\mathcal{S}(\mathbb{Z}^{N})$,
then by the discrete Hardy-Littlewood inequality in Lemma \ref{lem:rearrangement inequalities},
for any non-negative $u\in \ell_{b}^{q}(\mathbb{Z}^{N})$,
\[
\parallel\mu_{b}u\parallel_{\ell^{q}}\leq\parallel\mu_{b}u^{\star}\parallel_{\ell^{q}}.
\]
Hence, letting $\left\{ u_{n}\right\} $ be a minimizing sequence
satisfying (\ref{eq:min seq}), by Lemma \ref{lem:rearrangement inequalities} and Lemma \ref{lem:Polya-Szego_inequality}, we may assume that $u_{n}\in\mathcal{S}(\mathbb{Z}^{N})$.
Since $\left\{ u_{n}\right\} $ is uniformly bounded in $D_{0}^{1,p}(\mathbb{Z}^{N})$
and by the CKN inequality (\ref{eq:CKN inequality}), we have 
\[
\lVert\mu_{b}u_{n}\rVert_{\ell^{q^{\ast}}}\leq C.
\]
By Lemma \ref{lem:compact embedding}, since $q>q^{\ast}$, passing to a subsequence, 
\[
\mu_{b}u_{n}\to\mu_{b}u\quad\textrm{in}\;\ell^{q}
\]
for some $u\in\ell_{b}^{q}(\mathbb{Z}^{N}).$ Then $\lVert u\rVert_{\ell_{b}^{q}}=1$. Hence, $\lVert u\rVert_{D^{1,p}}=S$ and $u\in D_0^{1,p}(\mathbb{Z}^{N})$ by Theorem \ref{thm:the equivalence for Z^N_k=00003D00003D00003D1}. This proves the result.
\end{proof}

\begin{proof}[Proof of Theorem \ref{thm:main2}]
	Letting $\left\{ u_{n}\right\} $ be a minimizing sequence of
	(\ref{eq: inf-2}), by Lemma \ref{lem:rearrangement inequalities} and Lemma \ref{lem:Polya-Szego_inequality}, we may assume that $u_{n}\in\mathcal{S}(\mathbb{Z}^{N})$. Since
	$\parallel u_{n}\parallel_{D^{1,p}(\mathbb{Z}^{N})}^{\theta}\parallel u_{n}\parallel_{\ell^{r}(\mathbb{Z}^{N})}^{1-\theta}$
	is uniformly bounded, then by the CKN inequality (\ref{eq:CKN inequality}),
	we have 
	\[
    \lVert u_{n}\rVert_{\ell^{q^{\ast}}}\leq C.
    \]
	By Lemma \ref{lem:compact embedding}, since $q>q^{\ast}$,  passing to a subsequence, 
	\[
    u_{n}\to u\quad\textrm{in}\;\ell^{q}
     \]
	for some $u \in \ell^{q}(\mathbb{Z}^{N})$ and $\parallel u\parallel_{\ell^{q}}=1$.
	By the diagonal principle, passing to a subsequence we get the pointwise
	convergence
		\[
	u_{n}\to u\quad\textrm{pointwise on \ensuremath{\mathbb{Z}^{N}}}.
	\]
	Then for any $R>1$, on $B_{R}\coloneqq\left\{ x\in\mathbb{Z}^{N}:d(x)<R\right\} $,
	\[
	\parallel\nabla u_{n}\parallel_{\ell^{p}(B_{R})}^{\theta}\parallel u_{n}\parallel_{\ell^{r}(B_{R})}^{1-\theta}\to\parallel\nabla u\parallel_{\ell^{p}(B_{R})}^{\theta}\parallel u\parallel_{\ell^{r}(B_{R})}^{1-\theta},
	\]
	and by 
	\[
	\parallel\nabla u_{n}\parallel_{\ell^{p}(B_{R})}^{\theta}\parallel u_{n}\parallel_{\ell^{r}(B_{R})}^{1-\theta}\leq\varliminf\limits _{n\to\infty}\parallel\nabla u_{n}\parallel_{\ell^{p}(\mathbb{Z}^{N})}^{\theta}\parallel u_{n}\parallel_{\ell^{r}(\mathbb{Z}^{N})}^{1-\theta}=K,
	\]
	we get 
	\[
	\parallel\nabla u\parallel_{\ell^{p}(B_{R})}^{\theta}\parallel u\parallel_{\ell^{r}(B_{R})}^{1-\theta}\leq K.
	\]
	Letting $R\to+\infty$, 
	\begin{equation}
		\parallel\nabla u\parallel_{\ell^{p}(\mathbb{Z}^{N})}^{\theta}\parallel u\parallel_{\ell^{r}(\mathbb{Z}^{N})}^{1-\theta}\leq K.\label{eq: limit bounded}
	\end{equation}
	Then we claim that $u\in D_0^{1,p}(\mathbb{Z}^{N}) \cap \ell ^r(\mathbb{Z}^{N})$. Suppose by contradiction that $\parallel u\parallel_{\ell^{r}(\mathbb{Z}^{N})}=+\infty$,
	then $\parallel\nabla u\parallel_{\ell^{p}(\mathbb{Z}^{N})}=0$
	by (\ref{eq: limit bounded}). This yields $u\equiv \text{const}\neq0$ which
	contradicts to $\parallel u\parallel_{\ell^{q}(\mathbb{Z}^{N})}=1$.
	And if $\parallel\nabla u\parallel_{\ell^{p}(\mathbb{Z}^{N})}=+\infty$,
	then $\parallel u\parallel_{\ell^{r}(\mathbb{Z}^{N})}=0$, which contradicts
	to $\parallel u\parallel_{\ell^{q}(\mathbb{Z}^{N})}=1$. Hence, $u\in D_0^{1,p}(\mathbb{Z}^{N}) \cap \ell ^r(\mathbb{Z}^{N})$ and $u$ is a minimizer.
\end{proof}

By the proofs of Theorem \ref{thm:main1} and Theorem \ref{thm:main2}, we know that the minimizers obtained are Schwarz symmetric. Finally, we prove Corollary \ref{cor:1} and Corollary \ref{cor: GN_EL equation}. 
\begin{proof}[Proof of Corollary \ref{cor:1} and Corollary \ref{cor: GN_EL equation}]
	By Theorem \ref{thm:main1} there exists a Schwarz symmetric
	minimizer $u$ for $S$. It follows from the Lagrange multiplier that $u$ is a non-negative
	solution of equation (\ref{eq:EL equation}). The maximum principle yields that $u$ is
	positive. By the same argument, we can prove Corollary \ref{cor: GN_EL equation}.
\end{proof}

$\mathbf{\boldsymbol{\mathbf{Acknowledgements}\mathbf{}\text{:}}}$
The authors are grateful to Chao Ji for proposing the problem on discrete CKN inequalities, and thank him and Bobo Hua for helpful discussions and guidance. We also thank Hichem Hajaiej for his suggestions in an online discussion. We are grateful to Dong Ye for the  proofreading of the manuscript and many valuable comments and corrections.

\bibliographystyle{plain}
\bibliography{CKN_ref}

\begin{thebibliography}{10}

\bibitem{A76}
Thierry Aubin.
\newblock Probl\`emes isop\'{e}rim\'{e}triques et espaces de {S}obolev.
\newblock {\em J. Differential Geometry}, 11(4):573--598, 1976.

\bibitem{BPZ07}
Thomas Bartsch, Shuangjie Peng, and Zhitao Zhang.
\newblock Existence and non-existence of solutions to elliptic equations
  related to the {C}affarelli-{K}ohn-{N}irenberg inequalities.
\newblock {\em Calc. Var. Partial Differential Equations}, 30(1):113--136,
  2007.

\bibitem{BFV14}
Jacopo Bellazzini, Rupert~L. Frank, and Nicola Visciglia.
\newblock Maximizers for {G}agliardo-{N}irenberg inequalities and related
  non-local problems.
\newblock {\em Math. Ann.}, 360(3-4):653--673, 2014.

\bibitem{BE97}
Henri Berestycki and Maria~J. Esteban.
\newblock Existence and bifurcation of solutions for an elliptic degenerate
  problem.
\newblock {\em J. Differential Equations}, 134(1):1--25, 1997.

\bibitem{bollobas1991compressions}
B\'{e}la Bollob\'{a}s and Imre Leader.
\newblock Compressions and isoperimetric inequalities.
\newblock {\em J. Combin. Theory Ser. A}, 56(1):47--62, 1991.

\bibitem{BC98}
Ha\"{\i}m Brezis and Xavier Cabr\'{e}.
\newblock Some simple nonlinear {PDE}'s without solutions.
\newblock {\em Boll. Unione Mat. Ital. Sez. B Artic. Ric. Mat. (8)},
  1(2):223--262, 1998.

\bibitem{burchard2006rearrangement}
Almut Burchard and Hichem Hajaiej.
\newblock Rearrangement inequalities for functionals with monotone integrands.
\newblock {\em J. Funct. Anal.}, 233(2):561--582, 2006.

\bibitem{CKN84}
L.~Caffarelli, R.~Kohn, and L.~Nirenberg.
\newblock First order interpolation inequalities with weights.
\newblock {\em Compositio Math.}, 53(3):259--275, 1984.

\bibitem{CM99}
Paolo Caldiroli and Roberta Musina.
\newblock On the existence of extremal functions for a weighted {S}obolev
  embedding with critical exponent.
\newblock {\em Calc. Var. Partial Differential Equations}, 8(4):365--387, 1999.

\bibitem{CW01}
Florin Catrina and Zhi-Qiang Wang.
\newblock On the {C}affarelli-{K}ohn-{N}irenberg inequalities: sharp constants,
  existence (and nonexistence), and symmetry of extremal functions.
\newblock {\em Comm. Pure Appl. Math.}, 54(2):229--258, 2001.

\bibitem{CC93}
Kai~Seng Chou and Chiu~Wing Chu.
\newblock On the best constant for a weighted {S}obolev-{H}ardy inequality.
\newblock {\em J. London Math. Soc. (2)}, 48(1):137--151, 1993.

\bibitem{DDG11}
E.~N. Dancer, Yihong Du, and Zongming Guo.
\newblock Finite {M}orse index solutions of an elliptic equation with
  supercritical exponent.
\newblock {\em J. Differential Equations}, 250(8):3281--3310, 2011.

\bibitem{DL90}
Robert Dautray and Jacques-Louis Lions.
\newblock {\em Mathematical analysis and numerical methods for science and
  technology. {V}ol. 1}.
\newblock Springer-Verlag, Berlin, 1990.
\newblock Physical origins and classical methods, With the collaboration of
  Philippe B\'{e}nilan, Michel Cessenat, Andr\'{e} Gervat, Alain Kavenoky and
  H\'{e}l\`ene Lanchon, Translated from the French by Ian N. Sneddon, With a
  preface by Jean Teillac.

\bibitem{DD02}
Manuel Del~Pino and Jean Dolbeault.
\newblock Best constants for {G}agliardo-{N}irenberg inequalities and
  applications to nonlinear diffusions.
\newblock {\em J. Math. Pures Appl. (9)}, 81(9):847--875, 2002.

\bibitem{DE12-1}
Jean Dolbeault and Maria~J. Esteban.
\newblock About existence, symmetry and symmetry breaking for extremal
  functions of some interpolation functional inequalities.
\newblock In {\em Nonlinear partial differential equations}, volume~7 of {\em
  Abel Symp.}, pages 117--130. Springer, Heidelberg, 2012.

\bibitem{DE12-2}
Jean Dolbeault and Maria~J. Esteban.
\newblock Extremal functions for {C}affarelli-{K}ohn-{N}irenberg and
  logarithmic {H}ardy inequalities.
\newblock {\em Proc. Roy. Soc. Edinburgh Sect. A}, 142(4):745--767, 2012.

\bibitem{DETT11}
Jean Dolbeault, Maria~J. Esteban, Gabriella Tarantello, and Achilles Tertikas.
\newblock Radial symmetry and symmetry breaking for some interpolation
  inequalities.
\newblock {\em Calc. Var. Partial Differential Equations}, 42(3-4):461--485,
  2011.

\bibitem{MR4593125}
Rodrigo Duarte and Jorge Drumond~Silva.
\newblock Weighted {G}agliardo-{N}irenberg interpolation inequalities.
\newblock {\em J. Funct. Anal.}, 285(5):Paper No. 110009, 49, 2023.

\bibitem{frankl1989extremal}
P.~Frankl and Z.~F\"uredi.
\newblock Extremal problems whose solutions are the blowups of the small
  {W}itt-designs.
\newblock {\em J. Combin. Theory Ser. A}, 52(1):129--147, 1989.

\bibitem{GLY16}
Alexander Grigor'yan, Yong Lin, and Yunyan Yang.
\newblock Yamabe type equations on graphs.
\newblock {\em J. Differential Equations}, 261(9):4924--4943, 2016.

\bibitem{GHS22}
Qingsong Gu, Xueping Huang, and Yuhua Sun.
\newblock Superlinear elliptic inequalities on weighted graphs.
\newblock {\em arXiv:2201.06397}, 2022.

\bibitem{gupta2022symmetrization}
Shubham Gupta.
\newblock Symmetrization inequalities on one-dimensional integer lattice.
\newblock {\em arXiv preprint arXiv:2204.11647}, 2022.

\bibitem{HH10}
Hichem Hajaiej.
\newblock Rearrangement inequalities in the discrete setting and some
  applications.
\newblock {\em Nonlinear Anal.}, 72(3-4):1140--1148, 2010.

\bibitem{HHH22}
Hichem Hajaiej, Fengwen Han, and Bobo Hua.
\newblock Discrete schwarz rearrangement in lattice graphs.
\newblock {\em arXiv: 2209.01003}, 2022.

\bibitem{HL21}
Bobo Hua and Ruowei Li.
\newblock The existence of extremal functions for discrete {S}obolev
  inequalities on lattice graphs.
\newblock {\em J. Differential Equations}, 305:224--241, 2021.

\bibitem{HLM22}
Bobo Hua, Ruowei Li, and Florentin M\"unch.
\newblock Extremal functions for the second-order sobolev inequality on cayley
  graphs.
\newblock {\em Calc. Var. Partial Differential Equations}, 64(200):Paper No.
  200, 2025.

\bibitem{HLW22}
Bobo Hua, Ruowei Li, and Lidan Wang.
\newblock A class of semilinear elliptic equations on groups of polynomial
  growth.
\newblock {\em J. Differential Equations}, 363:327--349, 2023.

\bibitem{HX21}
Bobo Hua and Wendi Xu.
\newblock Existence of ground state solutions to some nonlinear {S}chr\"odinger
  equations on lattice graphs.
\newblock {\em Calc. Var. Partial Differential Equations}, 62(4):Paper No. 127,
  17, 2023.

\bibitem{HLY20}
An~Huang, Yong Lin, and Shing-Tung Yau.
\newblock Existence of solutions to mean field equations on graphs.
\newblock {\em Comm. Math. Phys.}, 377(1):613--621, 2020.

\bibitem{L22}
Ruowei Li.
\newblock The existence of positive ground state solutions for the {C}hoquard
  type equation on groups of polynomial growth.
\newblock {\em Discrete Contin. Dyn. Syst.}, 45(2):665--685, 2025.

\bibitem{LW22}
Ruowei Li and Lidan Wang.
\newblock The existence and convergence of solutions for the nonlinear choquard
  equations on groups of polynomial growth.
\newblock {\em Journal of Partial Differential Equations}, 2025.

\bibitem{L83}
Elliott~H. Lieb.
\newblock Sharp constants in the {H}ardy-{L}ittlewood-{S}obolev and related
  inequalities.
\newblock {\em Ann. of Math. (2)}, 118(2):349--374, 1983.

\bibitem{L86}
Chang~Shou Lin.
\newblock Interpolation inequalities with weights.
\newblock {\em Comm. Partial Differential Equations}, 11(14):1515--1538, 1986.

\bibitem{L84-1}
P.-L. Lions.
\newblock The concentration-compactness principle in the calculus of
  variations. {T}he locally compact case. {I}.
\newblock {\em Ann. Inst. H. Poincar\'{e} Anal. Non Lin\'{e}aire},
  1(2):109--145, 1984.

\bibitem{L85-2}
P.-L. Lions.
\newblock The concentration-compactness principle in the calculus of
  variations. {T}he limit case. {II}.
\newblock {\em Rev. Mat. Iberoamericana}, 1(2):45--121, 1985.

\bibitem{MP01}
\`E. Mitidieri and S.~I. Pokhozhaev.
\newblock A priori estimates and the absence of solutions of nonlinear partial
  differential equations and inequalities.
\newblock {\em Tr. Mat. Inst. Steklova}, 234:1--384, 2001.

\bibitem{NY22}
Quoc~Anh Ng\^{o} and Dong Ye.
\newblock Existence and non-existence results for the higher order
  {H}ardy-{H}\'{e}non equations revisited.
\newblock {\em J. Math. Pures Appl. (9)}, 163:265--298, 2022.

\bibitem{Ni82}
Wei~Ming Ni.
\newblock On the elliptic equation {$\Delta u+K(x)u^{(n+2)/(n-2)}=0$}, its
  generalizations, and applications in geometry.
\newblock {\em Indiana Univ. Math. J.}, 31(4):493--529, 1982.

\bibitem{Ni86}
Wei-Ming Ni.
\newblock Uniqueness, nonuniqueness and related questions of nonlinear elliptic
  and parabolic equations.
\newblock In {\em Nonlinear functional analysis and its applications, {P}art 2
  ({B}erkeley, {C}alif., 1983)}, volume~45 of {\em Proc. Sympos. Pure Math.},
  pages 229--241. Amer. Math. Soc., Providence, RI, 1986.

\bibitem{P07}
Joakim~H. Petersson.
\newblock Best constants for {G}agliardo-{N}irenberg inequalities on the real
  line.
\newblock {\em Nonlinear Anal.}, 67(2):587--600, 2007.

\bibitem{PS12}
Quoc~Hung Phan and Philippe Souplet.
\newblock Liouville-type theorems and bounds of solutions of
  {H}ardy-{H}\'{e}non equations.
\newblock {\em J. Differential Equations}, 252(3):2544--2562, 2012.

\bibitem{pruss98}
Alexander~R. Pruss.
\newblock Discrete convolution-rearrangement inequalities and the
  {F}aber-{K}rahn inequality on regular trees.
\newblock {\em Duke Math. J.}, 91(3):463--514, 1998.

\bibitem{RZ00}
Wolfgang Reichel and Henghui Zou.
\newblock Non-existence results for semilinear cooperative elliptic systems via
  moving spheres.
\newblock {\em J. Differential Equations}, 161(1):219--243, 2000.

\bibitem{RS09}
Grigori Rozenblum and Michael Solomyak.
\newblock On the spectral estimates for the {S}chr\"{o}dinger operator on
  {$\Bbb Z^d,\ d\ge3$}.
\newblock {\em J. Math. Sci. (N.Y.)}, 159(2):241--263, 2009.
\newblock Problems in mathematical analysis. No. 41.

\bibitem{shlapentokh2010asymptotic}
Yakov Shlapentokh-Rothman.
\newblock An asymptotic {F}aber-{K}rahn inequality for the combinatorial
  laplacian on $\mathbb{Z}^{2}$.
\newblock {\em arXiv preprint arXiv:1008.4092}, 2010.

\bibitem{T76}
Giorgio Talenti.
\newblock Best constant in {S}obolev inequality.
\newblock {\em Ann. Mat. Pura Appl. (4)}, 110:353--372, 1976.

\bibitem{WW20}
Z.-Q. Wang and M.~Willem.
\newblock Singular minimization problems.
\newblock {\em J. Differential Equations}, 161(2):307--320, 2000.

\bibitem{weinstein1999excitation}
M.~I. Weinstein.
\newblock Excitation thresholds for nonlinear localized modes on lattices.
\newblock {\em Nonlinearity}, 12(3):673--691, 1999.

\bibitem{ZZ18}
Ning Zhang and Liang Zhao.
\newblock Convergence of ground state solutions for nonlinear {S}chr\"{o}dinger
  equations on graphs.
\newblock {\em Sci. China Math.}, 61(8):1481--1494, 2018.

\bibitem{Z21}
Yang Zhang.
\newblock Optimizers of the {S}obolev and {G}agliardo-{N}irenberg inequalities
  in {$\dot W^{s,p}$}.
\newblock {\em Calc. Var. Partial Differential Equations}, 60(1):Paper No. 10,
  24, 2021.

\end{thebibliography}

\end{document}